\documentclass[journal]{IEEEtran}
\usepackage{xcolor,colortbl}
\definecolor{LightCyan}{rgb}{0.88,1,1}

\usepackage{graphics} 
\usepackage{epsfig} 
\usepackage{amsmath} 
\usepackage{amsthm,amsfonts}
\usepackage{amssymb}  

\usepackage{mathrsfs}

\usepackage{subfigure}
\usepackage{color} 
\usepackage{cite}
\usepackage{mathtools}
\usepackage{mathrsfs}
\usepackage{txfonts}
\usepackage{amssymb}
\usepackage{multirow}
\usepackage{subfigure}

\newcommand{\HEI}[1]{\bf Hybrid-EI}

\newtheorem{theorem}{Theorem}
\newtheorem{lemma}{Lemma}
\newtheorem{remark}{Remark}
\newtheorem{example}{Example}

\ifCLASSINFOpdf

\else

\fi

\begin{document}
%
\title{{Event-Triggered Stabilization of Linear Time-Delay Systems via Halanay-Type Inequality}}
%
%
%

\author{Kexue~Zhang 
\thanks{This work was supported by the Natural Sciences and Engineering Research Council of Canada (NSERC) under the grant {RGPIN-2022-03144}.}
\thanks{K. Zhang is with the Department
of Mathematics and Statistics, Queen's University, Kingston, Ontario K7L 3N6, Canada (e-mail: kexue.zhang@queensu.ca).}
 }

\maketitle
\thispagestyle{empty}
%
\begin{abstract}
This paper studies the event-triggered control problem for time-delay systems. A novel event-triggering scheme is proposed to exponentially stabilize a class of linear time-delay systems. By employing a new Halanay-type inequality and the Lyapunov function method, sufficient conditions on the design of control gain and selection of parameters in the proposed event-triggering scheme are derived to both ensure the exponential stability of the closed-loop system and exclude Zeno behavior. Two examples are given to demonstrate the effectiveness of the theoretical result.

\end{abstract}
%

\section{Introduction}\label{Sec1}

\IEEEPARstart{D}{ue} to the ubiquitous of time delays in the real world, time-delay systems have been investigated extensively (see, e.g.,~\cite{KG-JC-VLK:2003,DL-DVS:2011}), and various control methods have been successfully applied to time-delay systems (see, e.g.,~\cite{JPR:2003,EF:2014}). Among them, the event-triggered control approach has gained increasing attention in recent years (see, e.g.,~\cite{PT:2007,WPH-KHJ-PT:2012,CN-EG-JC:2019}). Compared with the conventional feedback control and sampled-data control, the method of event-triggered control updates the control input only when a certain measurement of the state error violates a well-designed threshold, and the control inputs remain unchanged between two consecutive control updates (see~\cite{WPH-KHJ-PT:2012}). Such a threshold violation is called an \emph{event}. Since the event times are determined explicitly by the event occurrence, one of the main challenges in the study of event-triggered control problems is to exclude \emph{Zeno behavior} (see~\cite{CN-EG-JC:2019}), a phenomenon of infinitely many event times over some finite time interval, from the control system.

Time-delay effects have been considered with the event-triggered control method in the following two fashions: 1) time delays in the event-triggered controllers; 2) time delays within the system dynamics. Most of the recent results focused on event-triggered controllers with time delays, and the system dynamics is free of delay effects (see, e.g.,~\cite{WZ-ZPJ:2015,AS-EF:2016,EN-PT-JC:2020}). Very few results have been reported on event-triggered control for time-delay systems. One of the major difficulties in this research area is to rule out Zeno behavior from the event-triggered control systems. A straightforward way to exclude Zeno behavior is to enforce a uniform lower bound of the \emph{inter-event times} which are the time periods between consecutive event times. For example, the event-triggered control method based on sampled-data information (see, e.g.,~\cite{LZ-ZW-HG-XL:2015,AB-PP:2021}) and periodic event-triggered control approach (see, e.g.,~\cite{LW-ZW-TH-GW:2015,AB-PP-IDL-MDF:2021}) both require certain lower bound of the inter-event times, and hence Zeno behavior can be excluded naturally. Another commonly used approach is the hybrid event-triggered control method. With the help of multiple control approaches, Zeno behavior can be excluded (see, e.g.,~\cite{ZF-CG-HG:2017,KZ-BG:2021}). It has been shown in~\cite{KZ-BG:2021} that a direct generalization of the results for delay-free systems from~\cite{PT:2007} to the event-triggered control problems of time-delay systems may fail to work. Hence, it is worthwhile to explore new event-triggering schemes for time-delay systems. Other than the above-mentioned methods, several novel event-triggering mechanisms have been successfully designed for some specific time-delay systems, and these results either require the memory of delayed system states (see, e.g.,~\cite{SD-NM-JFG:2014}) or apply to a certain particular type of system delays (see, e.g.,~\cite{QZ:2019,PZ-TL-ZPJ:2020}). Recently, an event-triggering algorithm was proposed in~\cite{KZ-BG-EB:2021} for the general nonlinear time-delay systems by the method of Lyapunov-Krasovskii functionals. Zeno behavior can be excluded if the function portion of the Lyapunov candidate is in the quadratic form, but a uniform lower bound of the inter-event times cannot be derived. The proposed event-triggering algorithm may not be applicable if such a lower bound does not exist, since physical hardware can only run with some finite frequencies. It can be seen that event-triggered control problems of time-delay systems deserve insightful explorations.

Halanay-type inequality introduced in~\cite[pp. 378]{AH:1966} is one of the widely used approaches for the stability analysis of time-delay systems. Numerous generalized Halanay-type inequalities have been constructed and applied to various control and stability analysis problems (see, e.g.,{~\cite{SM-KG:2000,FM-MM:2022,PP-2022} and references therein}). Therefore, it is worthwhile to investigate the event-triggered control problems of time-delay systems via Halanay-type inequalities. Motivated by the above discussions, we study the event-triggered control problem for a class of linear time-delay systems. We propose a new event-triggered control algorithm that activates control updates when a certain combination of system states and sampling error goes over an exponential threshold. It should be noted that the proposed algorithm is independent of system delays and hence simple to implement. To ensure the exponential stability of the closed-loop system, a novel Halanay-type inequality is introduced for the exponential convergence of non-negative functions to the origin. Compared with the existing Halanay-type inequalities (see, e.g., \cite{AH:1966}), the new one allows an exponential-convergent perturbation to the upper bound of the Dini derivative for the involved non-negative functions. Such a perturbation is essential to exclude Zeno behavior from the control system with the proposed event-triggering condition. By employing the Lyapunov function method and the Halanay-type inequality, sufficient conditions on the parameters in the event-triggering condition are derived, and the feedback control gain is derived to guarantee the closed-loop system is globally exponentially stable. Furthermore, sufficient conditions on the event-triggering parameters are also obtained to rule out Zeno behavior. Other than the major differences in the stability analysis method and the event-triggering algorithm, a detailed comparison with the existing results on event-triggered control for time-delay systems in~\cite{KZ-BG-EB:2021} can be found in Remark~\ref{Remark2}.

The rest of this letter is organized as follows. Section~\ref{Sec2} formulates the event-triggered control problem and introduces the event-triggering condition. A new Halanay-type inequality is introduced in Section~\ref{Sec3}, and a stabilization criterion is constructed in Section~\ref{Sec4}. We give two examples in Section~\ref{Sec5} to numerically illustrate the main result. Some conclusions are drawn in Section~\ref{Sec6}.

\emph{Notation.} Let $\mathbb{N}$ and $\mathbb{R}^+$ represent the set of non-negative integers and the set of non-negative reals. We use $\|x\|$ to denote the Euclidean norm of a vector $x\in\mathbb{R}^n$ and $\|A\|$ to represent the spectral norm of matrix $A\in\mathbb{R}^{n\times n}$.  For $n\times m$ matrix $B$, we use $B^\top$ to denote its transpose. For symmetric matrix $P\in\mathbb{R}^{n\times n}$, we use $P<0$ to indicate it is negative definite and $\lambda_{max}(P)$ to represent the largest eigenvalue of $P$. For a real-valued function $W:\mathbb{R}^+\rightarrow\mathbb{R}^+$, the upper Dini derivative is defined as follows:
\[
\mathrm{D}^+ W(t)=\limsup_{\Delta t\rightarrow 0^+} \frac{W(t+\Delta t)-W(t)}{\Delta t}.
\]

\section{Problem Formulation}\label{Sec2}
Consider the following linear control system with time-varying delay
\begin{align}\label{system}
\left\{\begin{array}{ll}
\dot{x}(t)=A_1 x(t) +A_2 x(t-\tau(t)) +B u(t),\cr
x(s)=\phi(s),~~s\in [-\bar{\tau},0],
\end{array}\right.
\end{align}
where $t\in\mathbb{R}^+$, matrices $A_1,A_2\in\mathbb{R}^{n\times n}$ and $B\in\mathbb{R}^{n\times m}$ are given; $u(t)=Kx(t)$ is the feedback control input, and $K\in \mathbb{R}^{m\times n}$ is the control gain to be designed; {continuous function $\tau(t)\in [0,\bar{\tau}]$ represents the bounded time-varying delay}, and $\bar{\tau}\geq 0$ is the maximum involved delay; initial function $\phi:[-\bar{\tau},0]\rightarrow \mathbb{R}^n$ is continuous. Control system~\eqref{system} with the sampled-data implementation is given as follows:
\begin{align}\label{ETCsystem}
\left\{\begin{array}{ll}
\dot{x}(t)=A_1 x(t) +A_2 x(t-\tau(t)) +B u(t),\cr 
u(t)=K x(t_k), ~t\in[t_k,t_{k+1}) \textrm{~for~} k\in \mathbb{N},
\end{array}\right.
\end{align}
where $\{t_{k}\}_{k\in\mathbb{N}}$ is the sequence of sampling times to be determined according to a certain triggering condition which will be introduced shortly. To do so, we define the sampling error 
\[
\epsilon(t)=x(t_k)-x(t),~t\in[t_k,t_{k+1}) \textrm{~for~} k\in \mathbb{N},
\]
then closed-loop system~\eqref{ETCsystem} can be written in terms of error~$\epsilon$:
\begin{align}\label{errorsystem}
\dot{x}(t)&=\left(A_1+BK\right) x(t) +A_2 x(t-\tau(t)) +BK \epsilon(t).
\end{align}
Consider Lyapunov function 
\[
V(t)=x^\top(t)Px(t)
\]
where $P\in\mathbb{R}^{n\times n}$ is {a symmetric positive definite matrix}. The sequence $\{t_k\}_{k\in\mathbb{N}}$ is determined in the following manner:
\begin{equation}\label{eventtime}
t_{k+1}=\inf\{t>t_k:~\mathcal{C}(x(t),\epsilon(t))\geq \zeta(t)\},
\end{equation}
where 
\[
\mathcal{C}(x(t),\epsilon(t))=2x^\top(t)PBK \epsilon(t)-\sigma x^\top(t)Px(t).
\]
and
\[
\zeta(t)=\alpha \|V_0\|_{\bar{\tau}} e^{-\beta t}
\] 
with $V_0(s)=x^\top(s)Px(s)$ for $s\in[-\bar{\tau},0]$ and $\|V_0\|_{\bar{\tau}}=\sup_{-\bar{\tau}\leq s\leq 0} \{V(s)\}$. Here, constants $\alpha,\beta>0$ and $\sigma\geq 0$. 

According to the definition of $t_{k+1}$ in~\eqref{eventtime}, the control input $u$ is updated when the \emph{event} $\mathcal{C}(x(t),\epsilon(t))\geq \zeta(t)$ occurs. The times in the sequence $\{t_k\}_{k\in\mathbb{N}}$ are called \emph{event times}, and the time difference between consecutive events are called inter-event times, that is, $\{t_{k+1}-t_k\}_{k\in\mathbb{N}}$. The sampled-data control in~\eqref{ETCsystem} with event times determined by~\eqref{eventtime} is called \emph{event-triggered control}. Since the event times are implicitly defined by~\eqref{eventtime}, accumulation of event times might occur. Such a phenomenon is called \emph{Zeno behavior}, that is, infinitely many events occur in a finite period. Therefore, it is essential to exclude Zeno behavior from the event-triggered control system~\eqref{ETCsystem}. 

The objective of this study is to design the feedback control gain $K$ and parameters $\alpha,\beta,\sigma$ in~\eqref{eventtime} so that the event-triggered control system~\eqref{ETCsystem} is both globally exponentially stable and free of Zeno behavior. If $\alpha=0$ in $\zeta$, then~\eqref{eventtime} reduces to the one discussed in~\cite{WPH-KHJ-PT:2012} for linear systems without delays. In Section~\ref{Sec4}, we will show that this exponential function is important to exclude Zeno behavior.

\section{Halanay-Type Inequality}\label{Sec3}
In this section, we introduce a new Halanay-type inequality which plays an important role in the design of the event-triggered control mechanism for the stabilization of system~\eqref{ETCsystem}.

\begin{lemma}\label{Lemma} 
Let $a,b,\alpha,\beta,r$ be positive constants, and $a>b+\alpha$. Let $v:[-r,\infty)\rightarrow \mathbb{R}^+$ be a continuous function and satisfy the following inequality
\begin{equation}\label{Halanay}
\mathrm{D}^+ v(t)\leq -a v(t) + b\sup_{-r\leq s\leq 0}\{v(t+s)\} + \alpha \|v_0\|_r e^{-\beta t}
\end{equation}
for $t\in\mathbb{R}^+$. Then the following inequality holds
{\[
v(t)\leq \|v_0\|_r e^{-\eta t} \textrm{~for~} t\in\mathbb{R}^+,
\]}
where $\|v_0\|_r=\sup_{-r\leq s\leq 0}\{v(s)\}$, {$\eta=\min\{\lambda,\beta\}$} and $\lambda$ is the real solution of the equation {
\[
a=b e^{\lambda r} +\lambda +\alpha.
\]}
\end{lemma}

\begin{proof}
Let 
{\[
w(t)=v(t)-\varepsilon \left(\|v_0\|_r +\xi \right) e^{-\eta t}
\] 
for some arbitrary $\varepsilon>1$ and $\xi>0$}. We will show $w(t)< 0$ by a contradiction argument. Suppose that there exists some $t>0$ such that $w(t)\geq 0$ and define 
\[
t^*=\inf\left\{ t\geq 0:~ w(t)\geq 0 \right\},
\] 
then $w(t^*)=0$ and $w(t)< 0$ for $0\leq t<t^*$, {that is,
\[
v(t^*)=\varepsilon \left(\|v_0\|_r +\xi \right) e^{-\eta t^*}
\]
and
\[
v(t)< \varepsilon \left(\|v_0\|_r +\xi \right) e^{-\eta t} \textrm{~for~} 0\leq t<t^*.
\]
For $-r\leq t<0$, we have
\[
v(t)\leq \sup_{-r\leq s<0}\{v(s)\} \leq \|v_0\|_r < \varepsilon \left(\|v_0\|_r +\xi \right) <  \varepsilon \left(\|v_0\|_r +\xi \right) e^{-\eta t}.
\]
Note that $e^{-\eta t}>1$ since $\eta$ is a positive constant and $t<0$.

Hence, we have 
\[
v(t)< \varepsilon \left(\|v_0\|_r +\xi \right) e^{-\eta t} \textrm{~for~all~} -r\leq t<t^*,
\]
and then,
\begin{align*}
\sup_{-r\leq s\leq 0}\left\{v(t^*+s)\right\}
&< \sup_{-r\leq s\leq 0}\left\{\varepsilon \left(\|v_0\|_r +\xi \right) e^{-\eta (t^*+s)}\right\} \cr
&\leq \varepsilon \left(\|v_0\|_r +\xi \right) e^{-\eta (t^*-r)} \cr
&= e^{\eta r}\varepsilon \left(\|v_0\|_r +\xi \right) e^{-\eta t^*}.
\end{align*}

From the above inequality and~\eqref{Halanay}, we have
\begin{align}\label{Diniw}
\mathrm{D}^+ w(t^*) &=\mathrm{D}^+ v(t^*) + \eta \varepsilon \left(\|v_0\|_r+\xi\right) e^{-\eta t^*} \cr
                    &\leq -a v(t^*) + b \sup_{-r\leq s\leq 0}\{v(t^*+s)\}+ \alpha \|v_0\|_r e^{-\beta t^*} \cr
                    &~~~+ \eta \varepsilon \left(\|v_0\|_r+\xi\right) e^{-\eta t^*}\cr
                    &< \left( -a+b e^{\eta r}+\eta \right) \varepsilon \left(\|v_0\|_r+\xi\right) e^{-\eta t^*}+ \alpha \|v_0\|_{r} e^{-\beta t^*} \cr
                    &< \left( -a+b e^{\eta r}+\eta + \alpha \right)\varepsilon \left(\|v_0\|_r+\xi\right) e^{-\eta t^*}. 
\end{align}
By the definition of $\lambda$ and the fact that $\eta\leq \lambda$, we get 
\[
-a+b e^{\eta r}+\eta + \alpha\leq 0,
\]}
and hence $\mathrm{D}^+ w(t^*)<0$ which implies $w$ is strictly decreasing in the interval $(t^*-\delta,t^*+\delta)$ for small enough $\delta>0$. This is a contradiction to the definition of $t^*$. Therefore, $w(t)<0$, that is, 
{\[
v(t)< \varepsilon (\|v_0\|_r+\xi) e^{-\eta t} \textrm{~for~} t\in\mathbb{R}^+.
\]}
{Letting $\varepsilon\rightarrow 1$ and $\xi\rightarrow 0$, we have 
\[
v(t)\leq \|v_0\|_r e^{-\eta t} \textrm{~for~} t\in\mathbb{R}^+.
\]}
This concludes the proof.
\end{proof}
\begin{remark}\label{Remark1}
If $\alpha=0$, then Lemma~\ref{Lemma} reduces to the Halanay-type inequality in~\cite{AH:1966} (see the lemma on page 378 of~\cite{AH:1966}). Hence, Lemma~\ref{Lemma} is less conservative than the conventional Halanay inequality. Moreover, the existence of this exponential function in~\eqref{Halanay} is essential in ruling out Zeno behavior from the event-triggered control system, which will be demonstrated in the following section. If $v(0)\not= 0$, then $\|v_0\|_{r}$ in~\eqref{Halanay} can be replaced by {$v(0)$} which simplifies the parameter selections in~\eqref{eventtime} as only the initial state is required instead of the entire initial function. {To be more specific, replacing $\|v_0\|_{r}$ by $v(0)$ in~\eqref{Halanay} will change the term $\alpha \|v_0\|_r e^{-\beta t^*}$ to $\alpha  v(0) e^{-\beta t^*}$ in the first inequality of~\eqref{Diniw}. By the fact that $v(0)\leq \|v_0\|_r$ we can see that the last inequality of~\eqref{Diniw} remains valid. Based on the above discussion, $\|v_0\|_{r}$ in~\eqref{Halanay} actually can be replaced by $v(s)$ for any $s\in[-r,0]$.} 
\end{remark}

{
\begin{remark}\label{Remark1.1}
Compared with the recent result in~\cite{FM-MM:2022}, Lemma~\ref{Lemma} has the following advantages. Lemma~\ref{Lemma} requires $a>b$ and then the parameter $\alpha$ in the triggering condition can be selected so that $a-b>\alpha$. However, the result in~\cite{FM-MM:2022} requires $a>b e^{a r}$, that is, the maximum involved delay $r$ is upper bounded by $\frac{1}{a}\ln(a/b)$. Therefore, Lemma~\ref{Lemma} can be applied to systems with any bounded time delays, while the result in~\cite{FM-MM:2022} is only applicable to systems with time delays smaller than $\frac{1}{a}\ln(a/b)$. Furthermore, the exponential convergence rate $\eta$ can be derived explicitly by solving the unique solution of the equation in Lemma~\ref{Lemma} and then comparing it with parameter $\beta$ in~\eqref{Halanay}.

\end{remark}}

\section{Stabilization Criterion}\label{Sec4}
In this section, we will employ the new Halanay-type inequality to derive the control gain $K$ and parameters $\alpha,\beta,\sigma$ in~\eqref{eventtime} to ensure exponential stability of~\eqref{ETCsystem} and also guarantee the non-existence of Zeno behavior.
\begin{theorem}\label{Th}
Suppose there exist positive constants $b$ and $h$, positive definite matrix $Q\in\mathbb{R}^{n\times n}$, and matrix $R\in\mathbb{R}^{m\times n}$ such that the following inequality
\begin{align}\label{LMI}
\begin{bmatrix}
QA_1^\top+A_1Q +R^\top B^\top +BR +(b+h)Q & A_2Q \\
QA_2^\top & -bQ 
\end{bmatrix}<0
\end{align}
is satisfied, then~system~\eqref{ETCsystem} is globally exponentially stable with $P=Q^{-1}$, $K=RP$, constants $\alpha$ and $\sigma$ selected so that $h>\alpha+\sigma$. Moreover, system~\eqref{ETCsystem} does not exhibit Zeno behavior for arbitrary $\beta>0$. If we further assume that $\beta\leq \lambda$ where $\lambda$ is the unique real solution of the equation 
\begin{equation}\label{lambda}
b+h-\sigma=b e^{\lambda \bar{\tau}} +\lambda +\alpha,
\end{equation}
then the inter-event times $\{t_{k+1}-t_k\}_{k\in\mathbb{N}}$ are uniformly lower bounded by a positive quantity $\Delta$, that is, $t_{k+1}-t_k\geq \Delta>0$ for all $k\in\mathbb{N}$.
\end{theorem}
\begin{proof}
Let $a=b+h-\sigma$, and consider Lyapunov function 
\[
V(t)=x^\top(t) P x(t)
\]
{ which is a continuous function of time $t$}. Along the trajectory of~\eqref{errorsystem}, we have
\begin{align}\label{dotV}
&\dot{V}(t)+a V(t) -b \sup_{-\bar{\tau}\leq s\leq 0}\{V(t+s)\}\cr
\leq& 2\dot{x}^\top(t) P x(t) +a x^\top(t) P x(t) -bx^\top(t-\tau(t))Px(t-\tau(t))\cr
 = &  x^\top(t)\left( (A_1+BK)^\top P+ P(A_1+BK) +aP \right) x(t) \cr
& +2x^\top(t)PA_2x(t-\tau(t)) + 2x^\top(t)PBK\epsilon(t)\cr
& -bx^\top(t-\tau(t))Px(t-\tau(t))\cr
 =&   \begin{bmatrix}
x^\top(t) & x^\top(t-\tau(t))
\end{bmatrix} \mathcal{S} \begin{bmatrix}
x^\top(t) \\
x^\top(t-\tau(t))
\end{bmatrix}+ 2x^\top(t)PBK\epsilon(t)\cr
\end{align}
where
\[
\mathcal{S}=\begin{bmatrix}
(A_1+BK)^\top P+ P(A_1+BK) +aP & P A_2 \\
A_2^\top P & -bP 
\end{bmatrix}.
\]
According to the event-triggering condition within~\eqref{eventtime}, we get
\begin{equation}\label{trigger}
2x^\top(t)PBK\epsilon(t) \leq \sigma x^\top(t)Px(t) +\alpha \|V_0\|_{\bar{\tau}} e^{-\beta t},
\end{equation}
then~\eqref{dotV} and~\eqref{trigger} imply 
\begin{align}\label{dotV1}
&\dot{V}(t)+a V(t) -b \sup_{-\bar{\tau}\leq s\leq 0}\{V(t+s)\}\cr
\leq& \begin{bmatrix}
x^\top(t) & x^\top(t-\tau(t))
\end{bmatrix} \bar{\mathcal{S}} \begin{bmatrix}
x^\top(t) \\
x^\top(t-\tau(t))
\end{bmatrix}+\alpha \|V_0\|_{\bar{\tau}} e^{-\beta t},
\end{align}
where
\[
\bar{\mathcal{S}}={\mathcal{S}}+ \begin{bmatrix}
 \sigma P & 0\cr
0 & 0
\end{bmatrix}.
\]
Multiply both sides of $\bar{\mathcal{S}}$ by 
\[
\begin{bmatrix}
P^{-1} & 0 \cr
0 & P^{-1}
\end{bmatrix},
\]
then~\eqref{LMI} implies $\bar{\mathcal{S}}<0$. Hence, we can conclude from~\eqref{dotV1} that
\begin{align}
\dot{V}(t)\leq -a V(t) +b \sup_{-\bar{\tau}\leq s\leq 0}\{V(t+s)\} +\alpha \|V_0\|_{\bar{\tau}} e^{-\beta t},
\end{align}
then it can be derived from Lemma~\ref{Lemma} that
\begin{equation}\label{Vbound}
V(t)\leq \|V_0\|_{\bar{\tau}} e^{-\eta t},~~t\in\mathbb{R}^+,
\end{equation}
where $\eta=\min\{\lambda,\beta\}$ and $\lambda$ is the solution of equation~\eqref{lambda}. {The global exponential stability of system~\eqref{ETCsystem} then follows from the definition of $V$ and~\eqref{Vbound}.}

Next, we will show system~\eqref{ETCsystem} with event times determined by~\eqref{eventtime} does not exhibit Zeno behavior. Integrating both sides of~\eqref{ETCsystem} from $t_k$ to $t$ for $t\in[t_k,t_{k+1})$ yields
\begin{align}\label{ebound}
&\|x(t)-x(t_k)\| \cr
=& \| \int^t_{t_k} \left[ A_1x(s) +A_2x(s-\tau(s)) +BKx(t_k) \right] \mathrm{d}s \|\cr
\leq& \int^t_{t_k} \left[ \|A_1\| \hskip.5mm \|x(s)\| +\|A_2\| \hskip.5mm \|x(s-\tau(s))\| +\|BK\| \hskip.5mm \|x(t_k)\| \right] \mathrm{d}s.  \cr
\end{align}
From~\eqref{Vbound}, we have
{
\begin{equation}\label{xbound}
\|x(t)\|\leq \sqrt{\frac{\|V_0\|_{\bar{\tau}}}{\lambda_{min}(P)}} e^{-\eta t/2}.
\end{equation} }
We then conclude from~\eqref{ebound} and~\eqref{xbound} that
{
\begin{align}\label{ebound1}
&\|x(t)-x(t_k)\| \cr 
\leq & \sqrt{\frac{\|V_0\|_{\bar{\tau}}}{\lambda_{min}(P)}}  \frac{2}{\eta} \left(\|A_1\| +\|A_2\|e^{\eta \bar{\tau}/2}\right)\left( e^{-\eta t_k/2} -e^{-\eta t/2} \right)\cr
     & + \sqrt{\frac{\|V_0\|_{\bar{\tau}}}{\lambda_{min}(P)}} \|BK\| (t-t_k) e^{-\eta t_k/2}
\end{align} }
Multiply both sides of~\eqref{ebound1} by $2\|x(t)\| \hskip.5mm \|PBK\|$, and we obtain with~\eqref{xbound} that
{
\begin{align}\label{ebound2}
&2\|x(t)\| \hskip.5mm \|PBK\| \hskip.5mm \|x(t)-x(t_k)\| \cr
\leq&  \delta_1 \|V_0\|_{\bar{\tau}} \left( e^{-\eta t_k/2} -e^{-\eta t/2} \right) e^{-\eta t/2}\cr
&+ \delta_2 \|V_0\|_{\bar{\tau}} (t-t_k) e^{-\eta t_k/2} e^{-\eta t/2}
\end{align} }
where
{
\[
\delta_1={\frac{\|PBK\|}{\lambda_{min}(P)}}  \frac{4}{\eta}  \left(\|A_1\| +\|A_2\|e^{\eta \bar{\tau}/2}\right)
\] }
and
{
\[
\delta_2=2{\frac{ \|PBK\| \hskip.5mm \|BK\|  }{\lambda_{min}(P)}}. 
\] }
According to the definition of event times in~\eqref{eventtime}, we have at $t=t_{k+1}$ that
{
\begin{align*}
&\|V_0\|_{\bar{\tau}} \left( \delta_1  \left( e^{-\eta t_k/2} -e^{-\eta t_{k+1}/2} \right) + \delta_2  (t_{k+1}-t_k) e^{-\eta t_k/2} \right) e^{-\eta t_{k+1}/2}  \cr
\geq& \sigma x^\top(t_{k+1})Px(t_{k+1}) +\alpha \|V_0\|_{\bar{\tau}} e^{-\beta t_{k+1}}\cr
\geq& \alpha \|V_0\|_{\bar{\tau}} e^{-\beta t_{k+1}}
\end{align*} }
which implies
{
\begin{align*}
&\delta_1  \left( e^{-\eta t_k/2} -e^{-\eta t_{k+1}/2} \right) e^{-\eta t_{k+1}/2}+ \delta_2  (t_{k+1}-t_k) e^{-\eta t_k/2} e^{-\eta t_{k+1}/2}\cr
\geq& \alpha e^{-\beta t_{k+1}}.
\end{align*} }
Multiply both sides of the above inequality by {$e^{\eta t_k}$} and define $T_k=t_{k+1}-t_k$, then we obtain
{
\begin{align}\label{contradiction}
\delta_1  \left( e^{-\eta T_k/2} -e^{-\eta T_{k}} \right) + \delta_2  T_k e^{-\eta T_k/2} \geq \alpha e^{-\beta T_{k}} e^{(\eta-\beta)t_k}.
\end{align} }

To show the event-triggered control system is free of Zeno behavior, we consider the following two scenarios.
\begin{itemize}
\item If $\beta>\lambda$, {then $\eta=\lambda$} and we will show Zeno behavior can be excluded by a contradiction argument. Suppose there exists a $\bar{t}<\infty$ such that {$\lim_{k\rightarrow\infty} t_k=\bar{t}$}. Then~\eqref{contradiction} implies 
{
\begin{align}\label{contradiction1}
\delta_1 \left( 1 -e^{-\eta T_{k}/2} \right) + \delta_2  T_k \geq \alpha e^{-(\beta-\eta/2) T_{k}} e^{(\eta-\beta)\bar{t}}
\end{align} }
for any $k\in\mathbb{N}$. Define a function
{
\[
g_1(T)=\alpha e^{-(\beta-\eta/2) T} e^{(\eta-\beta)\bar{t}}-\delta_1 \left( 1 -e^{-\eta T/2} \right) - \delta_2  T
\] }
then {$g_1(0)=\alpha e^{(\eta-\beta)\bar{t}}>0$}, $g_1(T)\rightarrow -\infty$ as $T\rightarrow\infty$, and $\dot{g}_1(T)<0$, that is, $g_1(T)$ strictly decreases from a positive quantity to $-\infty$ as $T$ goes to $\infty$. Hence, there exists a unique {$\bar{T}>0$} such that $g_1(\bar{T})=0$, and~\eqref{contradiction1} implies $t_{k+1}-t_k\geq \bar{T}$ for any $k\in\mathbb{N}$ which is a contradiction to the fact that $\lim_{k\rightarrow\infty} t_k=\bar{t}$. Therefore,~system~\eqref{ETCsystem} does not exhibit Zeno behavior.

\item If $\beta\leq \lambda$, {then $\eta=\beta$} and we can derive from~\eqref{contradiction} that
{
\begin{align}\label{contradiction2}
\delta_1 \left( 1 -e^{-\eta T_{k}/2} \right) + \delta_2  T_k \geq \alpha e^{-\eta T_{k}/2} 
\end{align} }
for any $k\in\mathbb{N}$. Define a function
{
\[
g_2(T)=\alpha e^{-\eta T/2} -\delta_1 \left( 1- e^{-\eta T/2} \right) - \delta_2  T,
\] }
and similar to the previous discussion, we can conclude from the facts $g_2(0)=\alpha>0$, $g_2(T)\rightarrow-\infty$ as $T\rightarrow\infty$, and $\dot{g}_2(T)<0$ for all $T\geq 0$ that there exists a unique solution $\tilde{T}>0$ to the equation $g_2(T)=0$. Hence, $T_k\geq \tilde{T}$ for all $k\in\mathbb{N}$, that is, the inter-event times $\{t_{k+1}-t_k\}_{k\in\mathbb{N}}$ have a uniform lower bound.
\end{itemize}
The proof of this theorem is complete.
\end{proof}

\begin{remark}\label{Remark2}
It can be observed from the proof that the exponential function~$\zeta(t)$ in~\eqref{eventtime} plays an important role in ruling out Zeno behavior. A similar exponential function has been considered in the event-triggering condition in~\cite{KZ-BG-EB:2021} for time-delay systems by using the method of Lyapunov-Krasovskii functionals. In this study, Lyapunov function method is employed with the help of a new Halanay-type inequality. The proposed event-triggering condition is also different from the one in~\cite{KZ-BG-EB:2021}. We enforce $2x^\top PBK \epsilon$
to be smaller than $\sigma x^\top P x+\zeta(t)$ in~\eqref{eventtime} which allows the norm of $\epsilon$ to be relatively large if $2x^\top PBK \epsilon$ is negative. However, the triggering condition in~\cite{KZ-BG-EB:2021} requires $\|\epsilon\|$ to be bounded. See Example~\ref{Example1} in Section~\ref{Sec5} for an illustration with Fig.~\ref{fig1}. Other than the above mentioned major differences, the conditions on the parameters in the exponential functions are distinct in the following sense. We require the convergence rate $\beta>0$ to be arbitrary and constant $\alpha$ small enough, while the result in~\cite{KZ-BG-EB:2021} restrict the convergence rate to be small enough so that Zeno behavior can be excluded. It should be noted that Zeno behavior can be excluded for any $\beta>0$ by Theorem~\ref{Th}, while a uniform lower bound of the inter-event times can be ensured if $\beta$ is small (i.e., $\beta\leq \lambda$). Such a uniform bound can not be guaranteed according to Theorem~3 of~\cite{KZ-BG-EB:2021} since the Lyapunov function is in a quadratic form.
\end{remark}

{
\begin{remark}
It is worth mentioning that system (1) reduces to the linear control system when $A_2=0$, then Theorem~\ref{Th} can be applied to the delay-free linear systems. Compared with the event-triggering rule in [6], our event-triggering condition (with positive $\alpha$) potentially postpones the occurrence of each event, and hence leads to larger inter-event times, that is, the control inputs are updated less frequently than the event-triggered control method in [6] for linear control systems.
\end{remark}
}

Next, we propose an algorithm for event-triggered stabilization of the linear time-delay system~\eqref{system}.

\hskip-4mm\rule{8.9cm}{0.5pt}
\vskip-0mm
\hskip-4mm\textbf{Event-triggered control algorithm.}
\vskip-1.5mm
\hskip-4mm\rule{8.9cm}{0.5pt}
\begin{itemize}
\item[I.] Prescribe parameters $b>0$ and $h>0$.

\item[II.] Solve the linear matrix inequality~\eqref{LMI} for $Q$ and $R$. Then, derive the feedback control gain 
\[
K=RQ^{-1}.
\]

\item[III.] Select parameters $\alpha>0$ and $\sigma>0$ such that $\alpha+\sigma<h$.
\end{itemize}
Event times can then be determined by~\eqref{eventtime} with $P=Q^{-1}$.
\vskip-1mm
\hskip-4mm\rule{8.9cm}{0.5pt}

\section{Examples}\label{Sec5}
In this section, two examples are provided to demonstrate the proposed event-triggering algorithm.

\begin{example}\label{Example1}
Consider the scalar system~\eqref{system} from~\cite{KZ-BG:2021,KZ-BG-EB:2021} with $x(t)\in\mathbb{R}$, $A_1=0$, $A_2=-0.1$, $\tau=16$, $B=1$, and $K=-0.2$. As indicated in~\cite{KZ-BG-EB:2021}, the control system is asymptotically stable. 

It can be verified that~\eqref{LMI} is satisfied with $b=0.1$, $h=0.2$, and $Q=R=1$. Select $\alpha=0.09$ and $\sigma=0.1$ so that $h>\alpha+\sigma$. Thus, we conclude from Theorem~\ref{Th} that system~\eqref{ETCsystem} with event times determined by~\eqref{eventtime} is globally exponentially stable, as illustrated in Fig.~\ref{fig1} with $\beta=0.11$ and $\phi(s)=1$ for $s\in[-\tau,0]$. It can be seen that control input was not updated over the time interval $(0,3.5)$ because $2xPBK\epsilon$ is negative. However, the input needs to be updated for multiple times over the same time interval according to the event-triggering scheme in~\cite{KZ-BG-EB:2021} as $|\epsilon|$ is required to be bounded (see Fig.~1(b) in~\cite{KZ-BG-EB:2021}). 

It has been shown in~\cite{KZ-BG:2021} that the scalar system~\eqref{system} exhibits Zeno behavior with the event-triggering algorithm generalized from~\cite{PT:2007}. {To rule out Zeno behavior}, the impulsive control approach is combined with the feedback control in~\cite{KZ-BG:2021} so that a uniform lower bound of the inter-event times can be guaranteed, and the Lyapunov function method with the Razumihin technique was employed to ensure asymptotic stability of the closed-loop system. With the help of the Halanay-type inequality in Lemma~\ref{Lemma}, we showed that scalar system~\eqref{system} in this example can be exponentially stabilized by the event-triggering algorithm proposed in Section~\ref{Sec4} without using impulsive controls.
\end{example}

\begin{figure}[!t]
\centering
\includegraphics[width=3.4in]{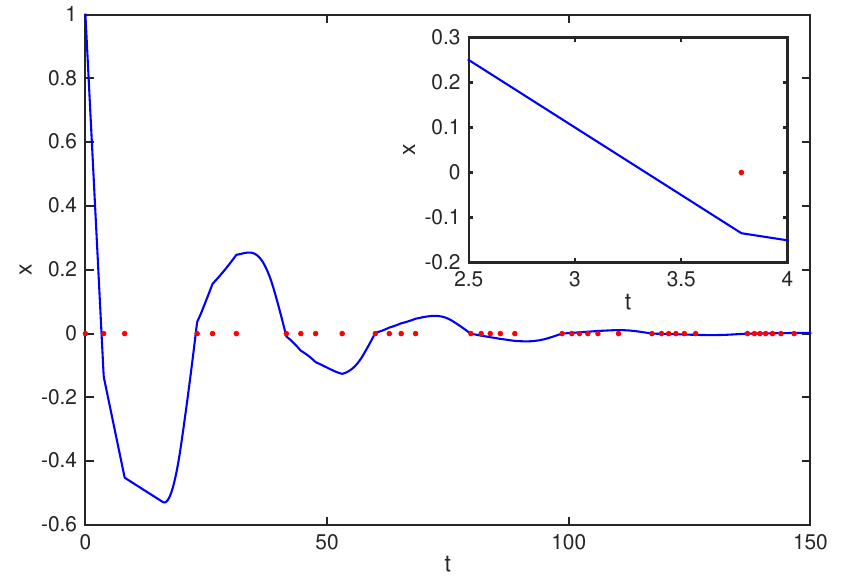}
\caption{Trajectory of system~\eqref{ETCsystem} in Example~\ref{Example1}. The red dots on the time axis indicate the event times determined by~\eqref{eventtime}.}
\label{fig1}
\end{figure}

\begin{example}\label{Example2}
Consider system~\eqref{system} with
\begin{align*}
A_1=\begin{bmatrix}
-1 & -0.5\\
3 & 2.5
\end{bmatrix},~A_2=\begin{bmatrix}
1.2 & 2\\
-0.4 & -1.2
\end{bmatrix},~B=\begin{bmatrix}
1\\
1
\end{bmatrix},
\end{align*}
and the fast-varying delay 
\[
\tau(t)=2-\sin(t^2),
\]
that is, the derivative of $\tau$ is unbounded. 

Next, we use the algorithm introduced in Section~\ref{Sec4} to design the event-triggered feedback control to exponentially stabilize system~\eqref{system}. 

Choose $b=1.1$ and $h=0.21$, then solving~\eqref{LMI} gives
\[
P=Q^{-1}=\begin{bmatrix}
1.5274 & 1.4575\\
1.4575 & 4.1300
\end{bmatrix} 
\]
and
\[
R=\begin{bmatrix}
-0.8221 & -0.7204
\end{bmatrix},
\] 
then $K=RP=\begin{bmatrix}
-2.3056 & -4.1733
\end{bmatrix}$. Select $\alpha=0.1$ and $\sigma=0.1$ such that $h>\alpha+\sigma$. Hence, Theorem~\ref{Th} concludes that system~\eqref{system} is globally exponentially stable with event times determined by~\eqref{eventtime}. {See Fig.~\ref{fig2} and Fig.~\ref{fig3} for numerical simulations of the event-triggered control system~\eqref{ETCsystem} with $\beta=1$ and different initial conditions.}

\end{example}

\begin{figure}[!t]
\centering
\includegraphics[width=3.4in]{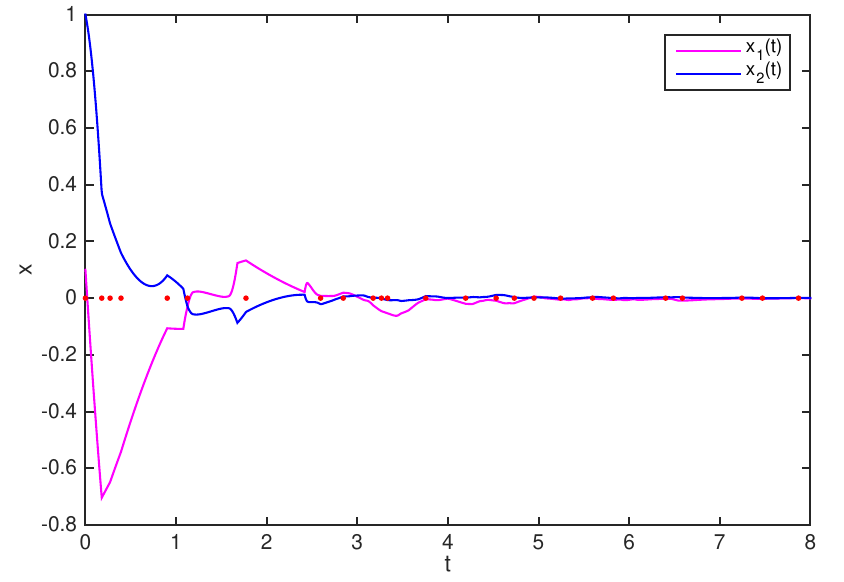}
\caption{Trajectories of system~\eqref{ETCsystem} in Example~\ref{Example2} with initial condition $\phi(s)=[0.1,~1]^\top$ for all $s\in[-3,0]$. Event times are indicated by the red dots on the time axis.}
\label{fig2}
\end{figure}

\begin{figure}[!t]
\centering
\includegraphics[width=3.4in]{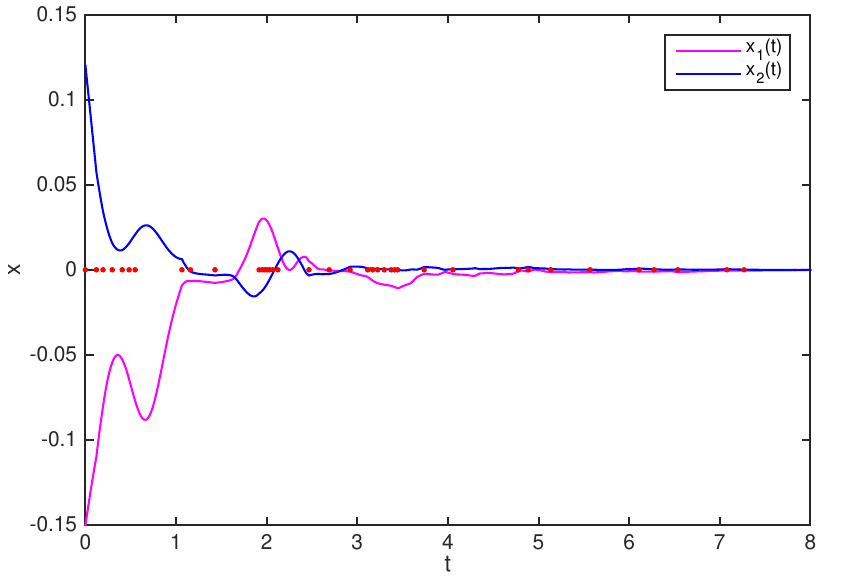}
\caption{Trajectories of system~\eqref{ETCsystem} in Example~\ref{Example2} with initial condition $\phi(s)=[-0.15\cos(3\pi s/2),~0.12\cos(\pi s)]^\top$ for all $s\in[-3,0]$. Event times are indicated by the red dots on the time axis.}
\label{fig3}
\end{figure}

\section{Conclusions}\label{Sec6}
In this paper, we have established a new Halanay-type inequality and designed a novel event-triggering scheme to exponentially stabilize
 the linear time-delay system. Sufficient conditions on the control gain and parameters in the event-triggering scheme have been derived to ensure exponential stability of the closed-loop system. Future research will focus on the extension of this work to general nonlinear systems with time delays and the study of time-delay effects in both the system dynamics and the control inputs.



\ifCLASSOPTIONcaptionsoff
  \newpage
\fi


%
%
%
%
%





\begin{thebibliography}{1}

\bibitem{KG-JC-VLK:2003}
K. Gu, J. Chen, and V. L. Kharitonov, \emph{Stability of Time-Delay Systems}. Springer Science \& Business Media, 2003.

\bibitem{DL-DVS:2011}
M. Lakshmanan and D. V. Senthilkumar, \emph{Dynamics of Nonlinear Time-Delay Systems}. Springer Science \& Business Media, 2011.

\bibitem{JPR:2003}
J. P. Richard, {``Time-delay systems: an overview of some recent advances and open problems"}. \emph{Automatica}, vol. 39, no. 10, pp. 1667-1694, 2003.

\bibitem{EF:2014}
E. Fridman, \emph{Introduction to Time-Delay Systems: Analysis and Control}. Basel, Switzerland: Birkh\"auser, 2014.

\bibitem{PT:2007}
P. Tabuada, ``Event-triggered real-time scheduling of stabilizing control tasks". \emph{IEEE Transactions on Automatic Control}, vol. 52, no. 2, pp. 1680-1685, 2007.

\bibitem{WPH-KHJ-PT:2012}
W. P. Heemels, K. H. Johansson, and P. Tabuada, {``An introduction to event-triggered and self-triggered control"}. In \emph{51st IEEE Conference on Decision and Control}, Maui, Hawaii, USA, pp. 3270-3285, 2012.

\bibitem{CN-EG-JC:2019}
C. Nowzari, E. Garcia, and J. Cort\'es, ``Event-triggered communication and control of networked systems for multi-agent consensus". \emph{Automatica}, vol. 105, pp. 1-27, 2019.

\bibitem{WZ-ZPJ:2015}
W. Zhang and Z.-P. Jiang, ``Event-based leader-following consensus of multi-agent systems with input time delay". \emph{IEEE Transactions on Automatic Control}, vol. 60, no. 5, pp. 1362-1367, 2015.

\bibitem{AS-EF:2016}
A. Selivanov and E. Fridman, ``Predictor-based networked control under uncertain transmission delays". \textit{Automatica}, vol. 80, pp. 101-108, 2016.

\bibitem{EN-PT-JC:2020}
E. Nozari, P. Tallapragada, and J. Cort\'es, ``Event-triggered stabilization of nonlinear systems with time-varying sensing and actuation delay". \textit{Automatica}, vol. 113, 108754, 2020.


\bibitem{LZ-ZW-HG-XL:2015}
L. Zou, Z. Wang, H. Gao, and X. Liu, ``Event-triggered state estimation for complex networks with mixed time delays via sampled data information: the continuous-time case". \emph{IEEE Transactions on Cybernetics}, vol. 45, no. 12, pp. 2804-2815, 2015.

\bibitem{AB-PP:2021}
A. Borri and P. Pepe, ``Event-triggered control of nonlinear systems with time-varying state delays". \textit{IEEE Transactions on Automatic Control}, vol. 66, no. 6, pp. 2846-2853, 2021.

\bibitem{LW-ZW-TH-GW:2015}
L. Wang, Z. Wang, T. Huang, and G. Wei, ``An event-triggered approach to state estimation for a class of complex networks with mixed time delays and nonlinearities". \emph{IEEE Transactions on Cybernetics}, vol. 46, no. 11, pp.2497-2508, 2015.

\bibitem{AB-PP-IDL-MDF:2021}
A. Borri, P. Pepe, I. Di Loreto, and M. Di Ferdinando, ``Finite-dimensional periodic event-triggered control of nonlinear time-delay systems with an application to the artificial pancreas". \textit{IEEE Control Systems Letters}, vol. 5, no. 1, pp. 31-36, 2021.

\bibitem{ZF-CG-HG:2017}
Z. Fei, C. Guan, and H. Gao, ``Exponential synchronization of networked chaotic delayed neural network by a hybrid event trigger scheme". \emph{IEEE Transactions on Neural Networks and Learning Systems}, vol. 29, no. 6, pp. 2558-2567, 2017.

\bibitem{KZ-BG:2021}
K. Zhang and B. Gharesifard, {``Hybrid event-triggered and impulsive control for time-delay systems"}. \emph{Nonlinear Analysis: Hybrid Systems}, vol. 43, 101109, 2021.

\bibitem{SD-NM-JFG:2014}
S. Durand, N. Marchand, and J.F. Guerrero-Castellanos, ``Event-based stabilization of nonlinear time-delay systems". In \textit{Proceedings of the 19th IFAC world congress}, Cape Town, South Africa, pp. 6953-6958, 2014.

\bibitem{QZ:2019}
Q. Zhu, ``Stabilization of stochastic nonlinear delay systems with exogenous disturbances and the event-triggered feedback control", \textit{IEEE Transactions on Automatic Control}, vol. 64, no. 9, pp. 3764-3771, 2019.

\bibitem{PZ-TL-ZPJ:2020}
P. Zhang, T. Liu, and Z.P. Jiang, ``Event-triggered stabilization of a class of nonlinear time-delay systems". \emph{IEEE Transactions on Automatic Control}, vol. 66, no. 1, pp. 421-428, 2020.

\bibitem{KZ-BG-EB:2021}
K. Zhang, B. Gharesifard, and E. Braverman, {``Event-triggered control for nonlinear time-delay systems"}. \emph{IEEE Transactions on Automatic Control}, vol. 67, no. 2, pp. 1031-1037, 2021.

\bibitem{AH:1966}
A. Halanay, \emph{Differential Equation: Stability, Oscillations, Time Lags}. vol. 6. New York, USA: Academic, 1966.

\bibitem{SM-KG:2000}
S. Mohamad and K. Gopalsamy, ``Continuous and discrete Halanay-type inequalities". \emph{Bulletin of the Australian Mathematical Society}, vol. 61, no. 3, pp. 371-385, 2000.

{
\bibitem{FM-MM:2022}
F. Mazenc and M. Malisoff, ``New versions of Halanay's inequality with multiple gain terms". \emph{IEEE Control Systems Letters}, vol. 6, pp. 1790-1795, 2022.
}


\bibitem{PP-2022}
P. Pepe, ``A nonlinear version of Halanay's inequality for the uniform convergence to the origin". \emph{Mathematical Control and Related Fields}, vol. 12, no. 3, pp. 789-811, 2022.





\end{thebibliography}
\end{document}